\documentclass[12pt]{amsart}
\usepackage{graphicx}
\usepackage{amssymb}
\usepackage{amsmath}
\usepackage{amscd}
\usepackage{upgreek}
\usepackage{mathrsfs}
\usepackage{stmaryrd}
\usepackage{longtable}
\usepackage{txfonts}
\usepackage{extarrows}
\usepackage[T1]{fontenc}
\usepackage{eucal}
\usepackage{arydshln}
\usepackage{wrapfig}
\usepackage{float}
\usepackage{xypic}
\usepackage{dsfont}
\usepackage{xcolor}
\usepackage{color}
\usepackage[colorlinks,linkcolor=blue,anchorcolor=blue,
citecolor=blue]{hyperref}
\usepackage{hyperref}
\usepackage[hmarginratio=1:1,top=32mm,columnsep=20pt]{geometry}
\usepackage{amssymb,color}
\definecolor{MyGreen}{RGB}{29,162,55}
	
\allowdisplaybreaks		
\usepackage{mathptmx}
\DeclareMathAlphabet{\mathcal}{OMS}{cmsy}{m}{n}
\DeclareSymbolFont{largesymbols}{OMX}{cmex}{m}{n}

\newtheorem{theorem}{Theorem}[section]
\newtheorem{prop}[theorem]{\rm \textsc{Proposition}}

\newtheorem{lem}[theorem]{\rm \textsc{Lemma}}
\newtheorem{coro}[theorem]{\rm \textsc{Corollary}}
\newtheorem{rem}[theorem]{\rm \textsc{Remark}}

\newtheorem{thm}[theorem]{\it \textsc{Theorem}}

\newcommand{\N}{\mathbb{N}} 
\newcommand{\F}{\mathbb{F}_{q}} 
\newcommand{\FF}{{\mathbb{F}}}

\newcommand{\ra}{\longrightarrow}

\newcommand{\GL}{{{\rm GL}(n, \F)}} 
\newcommand{\SL}{{{\rm SL}(n, \F)}} 
\newcommand{\U}{{{\rm U}(n, \F)}} 
\newcommand{\W}{{mW\oplus d W^{*}}}
\newcommand{\WW}{{W\oplus W^{*}}}

\newcommand{\diag}{{\rm diag}}
\newcommand{\id}{{\rm id}}

\newcommand{\alp}{\alphaup}
\newcommand{\bet}{\betaup}
\newcommand{\gam}{\gammaup}

\newcommand{\sig}{\sigmaup}

\begin{document}
\setlength{\oddsidemargin}{0cm}
\setlength{\evensidemargin}{0cm}

\title{On invariant fields of vectors and covectors}

\author{Yin Chen}
\address{School of Mathematics and Statistics, Northeast Normal University, Changchun 130024, P.R. China}
\email{ychen@nenu.edu.cn}

\author{David L. Wehlau}
\address{Department of Mathematics and Computer Science, Royal Military College, Kingston, ON, K7K 5L0, Canada}
\email{wehlau@rmc.ca}

\date{\today}
\def\shorttitle{On invariant fields of vectors and covectors}

\begin{abstract}
Let ${\mathbb{F}_{q}}$ be the finite field of order $q$.
Let $G$ be one of the three groups ${\rm GL}(n, \mathbb{F}_q)$, ${\rm SL}(n, \mathbb{F}_q)$ or  ${\rm U}(n, \mathbb{F}_q)$
 and let $W$ be the standard $n$-dimensional representation of $G$.
For non-negative integers $m$ and $d$ we let $mW\oplus d W^*$ denote the representation of $G$ given
 by the direct sum of $m$ vectors and $d$ covectors.
 We exhibit a minimal set
 of homogenous invariant polynomials $\{\ell_1,\ell_{2},\dots,\ell_{(m+d)n}\}\subseteq \mathbb{F}_q[mW\oplus d W^*]^G$ such that
$\mathbb{F}_q(mW\oplus d W^*)^G=\mathbb{F}_q(\ell_1,\ell_2,\dots,\ell_{(m+d)n})$ for all cases except when $md=0$ and
$G={\rm GL}(n, \mathbb{F}_q)$ or ${\rm SL}(n, \mathbb{F}_q)$.
 \end{abstract}

\subjclass[2010]{13A50.}
\keywords{Invariant fields; vector invariants; rationality.}

\maketitle
\baselineskip=20pt



\section{Introduction}

\setcounter{equation}{0}
\renewcommand{\theequation}
{1.\arabic{equation}}

\setcounter{theorem}{0}
\renewcommand{\thetheorem}
{1.\arabic{theorem}}


Let $\F$ be the finite field with $q$ elements and $\GL$ be the general linear group of degree $n\geqslant 1$.
Suppose $W$ is the standard representation of $\GL$ and $W^{*}$ is the dual space of $W$.
We denote by $\SL$ the special linear group, and  by $\U$  the unipotent group of upper triangular matrices with 1's on the diagonal.
Write $\W$ to denote the direct sum of $m$ copies of $W$ and $d$ copies of $W^{*}$ for $m,d\in\N$ and let
$G$ be one of the groups $\GL$, $\SL$ or $\U$.

We are concerned with the invariant rings
$\FF[\W]^{G}=\{f\in \FF[\W]\mid \sig(f)=f,\forall \sig\in G\}$ and the invariant fields
 $\FF(\W)^{G}=\{f\in \FF(\W)\mid \sig(f)=f,\forall \sig\in G\}$.
 In particular, the purpose of this paper is to study the following question:
\begin{equation} \label{question}\tag{$\ast$}
\begin{split}
&\textit{Does there exist }\ell_1,\ell_2,\dots,\ell_{(m+d)n}\in \F[\W]^{G}\textit{ such that }\\
&\quad\quad\quad\quad\quad\F(\W)^{G}=\F(\ell_1,\ell_2,\dots,\ell_{(m+d)n})~?
 \end{split}
 \end{equation}
Here we will show that the answer to Question~(\ref{question}) is positive for $\F[\W]^G$ and $G \in \{\GL, \SL, \U\}$ in all cases except possibly for the cases where $G=\GL$ or $\SL$ and either $m$ or $d$ is 0.


An affirmative answer to Question~(\ref{question}) for $\F(W)^{\GL}$ and $\F(W)^{\SL}$ dates back to
Dickson \cite{Dic1911} who proved that both rings of invariants $\F[W]^{\GL}$ and $\F[W]^{\SL}$ are polynomial algebras.
 In 1975, Mui \cite{Mui1975} proved that the invariant ring
$\F[W]^{\U}$ is also a polynomial algebra, thus  showing that Question~(\ref{question}) has a positive answer for $\F(W)^{\U}$.

  These results are especially significant in light of the so called "No-name Lemma" which we now state, see for example Jensen-Ledet-Yui \cite[Section 1.1, page 22]{JLY2003}.

\begin{lem}[No-Name Lemma]
Let $G$ be a finite group acting faithfully on a finite dimensional vector space $V$ over a field $\FF$, and let $U$ be a faithful
$\FF(G)$-submodule of $V$. Then the extension of invariant fields $\FF(V)^{G}/\FF(U)^{G}$ is purely transcendental.
\end{lem}


  Thus the results of Dickson and Mui imply that each of the fields $\F(\W)^{G}$ is
purely transcendental over the base field $\F$ for each of the groups $G=\GL$, $\SL$ or $\U$.

It is well-known that for any finite $p$-group $P$ and any modular representation  $V$, the invariant field is always purely transcendental, see Miyata \cite[Theorem 1]{Miy1971} or Kang \cite[Theorem]{Kan2006}. In 2007, Campbell-Chuai \cite[Theorem 2.4]{CC2007} gave an inductive method to find a polynomial generating set for the invariant field of $P$.  In particular this showed that $\F(\W)^{\U}$
can be generated by $(m+d)n$ homogeneous invariant polynomials. However, Campbell-Chuai \cite{CC2007} did not give any explicit polynomial generating sets for $\F(\W)^{\U}$ and for all $m,d,n$.

For a special  case where $n=2$ and $q=p$ is prime,
 Richman \cite[Corollary, page 38]{Ric1990} exhibited a polynomial generating set for
 $\mathbb{F}_{p}(mW)^{\mathbb{U}(2,\mathbb{F}_{p})}$, which is used to study the invariant ring $\mathbb{F}_{p}[mW]^{\mathbb{U}(2,\mathbb{F}_{p})}$, a classical object of study in modular invariant theory.

The structure of $\mathbb{F}_{p}[mW]^{\mathbb{U}(2,\mathbb{F}_{p})}$ is more complicated than that of $\mathbb{F}_{p}(mW)^{\mathbb{U}(2,\mathbb{F}_{p})}$, see for example, Campbell-Hughes \cite{CH1997}, Campbell-Shank-Wehlau \cite{CSW2010}, Wehlau \cite{Weh2013}, and Campbell-Wehlau \cite{CW2014}.


In 2011, Bonnaf\'e-Kemper \cite{BK2011} initiated a study of the modular invariant ring of a vector and a covector, showing that the invariant ring $\F[W\oplus W^{*}]^{\U}$ is a complete intersection algebra as well as raising  a conjecture on generating sets for
of $\F[W\oplus W^{*}]^{\GL}$. Recently, Chen-Wehlau \cite{CW2015} confirmed this conjecture of Bonnaf\'e-Kemper's.
As the first step of the proof of this conjecture, we proved that $\F(W\oplus W^{*})^{\GL}$ can be generated by $2n$ homogeneous invariant polynomials from $\F[W\oplus W^{*}]^{\GL}$, answering the Question~(\ref{question}) affirmatively in this special case.

The purpose of this paper is to generalize the above results to the invariant fields of any number of vectors and covectors.
Our main result is the following Theorem~\ref{mt} whose proof will be separated into three parts below: Theorems~\ref{GL},~\ref{SL}, and \ref{UU}.

\begin{thm} \label{mt}
For all $m,d\in \mathbb{N}^{+}$ and $G\in \{\GL,\SL,\U\}$ with standard representation $W$, there exist $(m+d)n$ invariant polynomials
$\ell_{1},\ell_{2},\dots,\ell_{(m+d)n}\in \F[\W]^{G}$ such that
$\F(\W)^{G}=\F(\ell_{1},\ell_{2},\dots,\ell_{(m+d)n})$.
\end{thm}
  Furthermore, we will explicitly exhibit such invariant polynomials  $\ell_{1},\ell_{2},\dots,\ell_{(m+d)n}$ generating the field of invariants in each case.

\begin{rem}
{\rm
By Campbell-Chuai  \cite[Theorem 2.4]{CC2007}, we see that Question~(\ref{question}) for $\F(mW)^{\U}$ has an affirmative answer.
}
\end{rem}

\section{Special Example: $m=d=1$}

\setcounter{equation}{0}
\renewcommand{\theequation}
{2.\arabic{equation}}

\setcounter{theorem}{0}
\renewcommand{\thetheorem}
{2.\arabic{theorem}}

In this section we prove  Theorem~\ref{mt} in the particular case $m=d=1$.  In the next section we will
give proofs for the general cases using the result for this special case.

\subsection{Generators of the invariant fields $\F(\WW)^{\GL}$ and $\F(\WW)^{\SL}$}

We choose  $\{x_{1},x_{2},\dots,x_{n}\}$ as a basis of $W$ and $\{y_{1},y_{2},\dots,y_{n}\}$ as the dual basis in $W^{*}$. Then
$$\F[\WW]=\F[x_{1},x_{2},\dots,x_{n},y_{1},y_{2},\dots,y_{n}]$$ is a polynomial algebra of Krull dimension $2n$, which can be endowed with an involution which is an algebra endomorphism
$$*:\F[\WW]\ra\F[\WW], ~~f\mapsto f^{*}$$  determined by
$x_{1}\mapsto y_{n}, x_{2}\mapsto y_{n-1},\dots,x_{n-1}\mapsto y_{2},x_{n}\mapsto y_{1}.$
There  are also  two $\F$-algebra homomorphisms:
\begin{eqnarray*}
F:\F[\WW]\ra\F[\WW] & \textrm{by} & x_{i}\mapsto x_{i}^{q}, y_{i}\mapsto y_{i} \\
F^{*}:\F[\WW]\ra\F[\WW] & \textrm{by} & x_{i}\mapsto x_{i}, y_{i}\mapsto y_{i}^{q}.
\end{eqnarray*}
We have a natural invariant $u_0$ in $\F[\WW]^{\GL}$ corresponding to the natural pairing between $W$ and $W^*$:
\begin{equation*}
\label{ }
u_{0}:=x_{1}y_{1}+x_{2}y_{2}+\cdots+x_{n}y_{n}.
\end{equation*}
For $i\in\N^{+}$, we define
\begin{eqnarray*}
u_{i} & := & F^{i}(u_{0}) = x_{1}^{q^{i}}y_{1}+x_{2}^{q^{i}}y_{2}+\cdots+x_{n}^{q^{i}}y_{n}\\
u_{-i} & := & (F^{*})^{i}(u_{0})=x_{1}y_{1}^{q^{i}}+x_{2}y_{2}^{q^{i}}+\cdots+x_{n}y_{n}^{q^{i}}
\end{eqnarray*}
which are also $\GL$-invariants, since $F$ and $F^{*}$ commute with the action of $\GL$. We observe that
$u_{-i}^{*}=u_{i}$ for all $i\in\N$.

Suppose $\F[W]^{\GL}=\F[c_{0},c_{1},\dots,c_{n-1}]$ and $\F[W]^{\SL}=\F[d_{n},c_{1},\dots,c_{n-1}]$ denote the Dickson invariants.
Note that $c_{0}=d_{n}^{q-1}$. We write $\F[W]^{\U}=\F[f_{1},f_{2},\dots,f_{n}]$ for the Mui invariants.
See Bonnaf\'e-Kemper \cite{BK2011} or Chen-Wehlau \cite{CW2015}  for details.

\begin{prop}[Chen-Wehlau  \cite{CW2015} (Proposition~5)]\label{pGL}
\begin{eqnarray*}
\F(\WW)^{\GL}&=&\F(c_{0},u_{1-n},u_{2-n},\dots,u_{0},u_{1},\dots,u_{n-1})\\
&=&\F(c_{0}^{*},u_{1-n},u_{2-n},\dots,u_{0},u_{1},\dots,u_{n-1})\\
&=&\F(c_{0},c_{1},c_2,\dots,c_{n-1},u_{0},u_{1},\dots,u_{n-1})\ .
\end{eqnarray*}
\end{prop}

We can use the above result to compute the corresponding field for the action $\SL$ on $\WW$.

\begin{prop}\label{pSL}
\begin{eqnarray*}
\F(\WW)^{\SL}&=&\F(d_{n},u_{1-n},u_{2-n},\dots,u_{0},u_{1},\dots,u_{n-1})\\
&=&\F(d_{n}^{*},u_{1-n},u_{2-n},\dots,u_{0},u_{1},\dots,u_{n-1})\\
&=&\F(d_{n},c_{1},\dots,c_{n-1},u_{0},u_{1},\dots,u_{n-1})\ .
\end{eqnarray*}
\end{prop}

\begin{proof}
  Put $K := \F(c_0,u_{1-n},u_{2-n},\dots,u_{0},u_{1},\dots,u_{n-1})=\F(\WW)^{\GL}$
  and $L := K(d_n)$.
Since $d_n$ is $\SL$-invariant we have $K(d_n) \subseteq M:=\F(\WW)^{\SL}$.
Thus $$K=\F(\WW)^{\GL} \subseteq L=K(d_n) \subseteq M.$$

  Since $c_0 = d_n^{q-1}$ we have
$L= K(d_n) = \F(d_{n},u_{1-n},u_{2-n},\dots,u_{0},u_{1},\dots,u_{n-1})$.
  Furthermore, it is clear that the minimal polynomial of
$d_n$ over $\F[\WW]^{\GL}$ is $X^{q-1} - c_0$.  Thus $[L:K]=q-1=[\GL:\SL]=[M:K]$.
Therefore $L=M$ proving the first equality.

  The other two equalities follow similarly.
\end{proof}

\begin{rem}{\rm
  In the proof of Theorem~\ref{GL} below  Equation~(\ref{big equation}) explicitly shows that
  $d_n d_n^* \in \F[u_{1-n},u_{2-n},\dots,u_{0},u_{1},\dots,u_{n-1}]$.}
\end{rem}

\begin{rem}{\rm
  It is also possible to prove Proposition~\ref{pSL} using the proof of
Proposition~\ref{pGL} from  \cite{CW2015} but replacing $c_0$ by $d_n$ and
 $c_0^*$ by $d_n^*$ throughout.}
\end{rem}

\subsection{$\F(\WW)^{\U}$}

The rest of this section is devoted to finding a minimal generating set for $\F(\WW)^{\U}$.

\begin{prop}\label{pU}\ \\
\begin{enumerate}
  \item  For $n=1$, $\F(\WW)^{\U}=\F(\WW)$. 
  \item  For $n=2$, $\F(\WW)^{\U}=\F(f_{1},f_{1}^{*},f_{2}^{*},u_{0})=\F(f_{1},f_{1}^{*},f_{2},u_{0})$. 
  \item\label{third case of pU} For any $n\geqslant 3$, the invariant field is given by
\begin{eqnarray*}
\F(\WW)^{\U}& = & \F(f_{1},f_{2},f_{1}^{*},f_{2}^{*},\dots,f_{n-2}^{*},u_{1},u_{0},u_{-1},\dots,u_{2-n}) \\
 & = &  \F(f_{1}^{*},f_{2}^{*},f_{1},f_{2},\dots,f_{n-2},u_{-1},u_{0},u_{1},\dots,u_{n-2})\ .
\end{eqnarray*}
\end{enumerate}
\end{prop}

\begin{proof}
(1) Since $\textrm{U}(1,\F)=\{1\}$ is the trivial group, $\F(\WW)^{\textrm{U}(1,\F)}=\F(x_{1},y_{1})^{\textrm{U}(1,\F)}=\F(x_{1},y_{1})=\F(f_{1},f_{1}^{*})$.

(2) By Bonnaf\'e-Kemper \cite[Equation (2.11), page 106]{BK2011},
the invariant ring $\F[\WW]^{\textrm{U}(2,\F)} = \F[u_{0},f_{1},f_{2},f_{1}^{*},f_{2}^{*}]$, is a hypersurface ring with the only relation
\begin{equation*}
\label{ }
u_{0}^{q}-(f_{1}f_{1}^{*})^{q-1}u_{0}-f_{1}^{q}f_{2}^{*}-f_{1}^{*q}f_{2}=0.
\end{equation*}
Thus $\F(\WW)^{\textrm{U}(2,\F)}=\F(f_{1},f_{1}^{*},f_{2}^{*},u_{0})=\F(f_{1},f_{1}^{*},f_{2},u_{0})$.

(3) First of all, we rewrite the relations
($R_{1}^{+}$), ($R_{2}$), ($R_{3}^{-}$), ($R_{3}$),  ($R_{4}^{-}$), ($R_{4}$), $\dots$,  ($R_{n-1}^{-}$), ($R_{n-1}$) and  ($R_{n}^{-}$), in Bonnaf\'e-Kemper \cite[page 105]{BK2011} as follows:
\begin{align*}
\tag{$R_{1}^{+}$} \label{R1+}  f_{1}^{q}\cdot f_{n}^{*}&= u_{2-n}^{q}+\alp_{3-n}\cdot u_{3-n}^{q}+\dots+ \alp_{0}\cdot u_{0}^{q}+\alp_{1}\cdot u_{1}, \\
 & \quad\textrm{where all } \alp_{k}\in\F[f_{1}^{*},\dots,f_{n-1}^{*}]\textrm{ are non-zero}.\\
\tag{$R_{2}$} \label{R2}  f_{2}\cdot f_{n-1}^{*}&= u_{2-n}+\alp_{3-n}\cdot u_{3-n}+\dots+ \alp_{0}\cdot u_{0}\\
 &\quad + u_{3-n}^{q}+\bet_{4-n}\cdot u_{4-n}^{q}+\dots+ \bet_{0}\cdot u_{0}^{q}+ \bet_{1}\cdot u_{1}, \\
 &\quad  \textrm{where all } \alp_{k},\bet_{k}\in\F[f_{1},f_{1}^{*},\dots,f_{n-2}^{*}] \textrm{ are non-zero}. \\
 \tag{$R_{3}^{-}$} \label{R3-}  f_{3}\cdot f_{n-2}^{*q}&= u_{2-n}+\alp_{3-n}\cdot u_{3-n}+\dots+ \alp_{-1}\cdot u_{-1}\\
 &\quad + u_{3-n}^{q}+\bet_{4-n}\cdot u_{4-n}^{q}+\dots+ \bet_{0}\cdot u_{0}^{q}\\
  &\quad + u_{4-n}^{q^{2}}+\gam_{5-n}\cdot u_{4-n}^{q^{2}}+\dots+ \gam_{0}\cdot u_{0}^{q^{2}}+ \gam_{1}\cdot u_{1}^{q},\\
 &\quad  \textrm{where all } \alp_{k},\bet_{k}, \gam_{k}\in\F[f_{1},f_{2},f_{1}^{*},\dots,f_{n-3}^{*}] \textrm{ are non-zero}.\\
 \tag{$R_{3}$} \label{R3}  f_{3}\cdot f_{n-2}^{*}&= u_{3-n}+\alp_{4-n}\cdot u_{4-n}+\dots+ \alp_{0}\cdot u_{0}\\
 &\quad + u_{4-n}^{q}+\bet_{5-n}\cdot u_{5-n}^{q}+\dots+ \bet_{0}\cdot u_{0}^{q}+  \bet_{1}\cdot u_{1}\\
  & \quad + u_{5-n}^{q^{2}}+\gam_{6-n}\cdot u_{6-n}^{q^{2}}+\dots+ \gam_{0}\cdot u_{0}^{q^{2}}+ \gam_{1}\cdot u_{1}^{q}+ \gam_{2}\cdot u_{2},\\
 & \quad \textrm{where all } \alp_{k},\bet_{k}, \gam_{k}\in\F[f_{1},f_{2},f_{1}^{*},\dots,f_{n-3}^{*}] \textrm{ are non-zero}.\\
 &\vdots\qquad \qquad \qquad\vdots\qquad \qquad \qquad\vdots\\
  &\vdots\qquad \qquad \qquad\vdots\qquad \qquad \qquad\vdots\\
  \tag{$R_{n}^{-}$} \label{Rn-}   f_{1}^{*q}\cdot f_{n}&=  u_{n-2}^{q}+\alp_{n-3}\cdot u_{n-3}^{q}+\dots+\alp_{0}\cdot u_{0}^{q}+\alpha_{-1}\cdot u_{-1}, \\
 &  \quad  \textrm{where all } \alp_{k}\in\F[f_{1},\dots,f_{n-1}]  \textrm{ are non-zero}.
\end{align*}

Define $L:=\F(f_{1},f_{2},f_{1}^{*},f_{2}^{*},\dots,f_{n-2}^{*},u_{1},u_{0},u_{-1},\dots,u_{2-n})$. Clearly, $L\subseteq \F(\WW)^{\U}$ and $L$ is generated by
$2n$ polynomials. The transcendence degree of $\F(\WW)^{\U}$ is equal to $2n$. Thus it suffices to show that
 $\F(\WW)^{\U}\subseteq L$.
By Bonnaf\'e-Kemper \cite[Proposition 1.1]{BK2011}) we have $\F(\WW)^{\U}=\F(f_{1},\dots,f_{n},f_{1}^{*},\dots,f_{n}^{*},u_{0})$.
Hence, it suffices to show
\begin{equation}
\label{ }
f_{n-1}^{*},f_{n}^{*},f_{3},\dots,f_{n}\in L
\end{equation}
which we will show using the above relations. Indeed,
it follows from the relation (\ref{R2}) that
\begin{equation}
\label{ }
f_{n-1}^{*}\in L\ .
\end{equation}
Thus (\ref{R1+}) implies that
\begin{equation}
\label{ }
f_{n}^{*}\in L.
\end{equation}
The relation (\ref{R3-}) yields
\begin{equation}
\label{ }
f_{3}\in L
\end{equation}
which combining with (\ref{R3}) gives
\begin{equation}
\label{ }
u_{2}\in L.
\end{equation}
The fact that $u_{2}\in L$ together with $(R_{4}^{-})$ shows
that $f_{4}\in L$, and $f_{4}\in L$, together with $(R_{4})$, implies that $u_{3}\in L$.
Proceeding in this way and in this order, we deduce $f_{k}\in L$ from $(R_{k}^{-})$ and then using $(R_{k})$ we conclude
\begin{equation}
\label{ }
u_{k-1}\in L
\end{equation}
which in turn serves to show that
\begin{equation}
\label{ }
f_{k+1}\in L.
\end{equation}
Continuing we finally have
$
f_{3},\dots,f_{n},u_{2},\dots,u_{n-2}\in L.
$
This finishes the proof of the first equation of Proposition~\ref{pU}~(\ref{third case of pU}).

By Bonnaf\'e-Kemper \cite[Example 2.1]{BK2011},  $\U$ is a $*$-stable subgroup
of $\GL$. We apply the involution $*$ to $\F(\WW)^{\U}$ to demonstrate the second equation.
\end{proof}

\section{General Cases}

\setcounter{equation}{0}
\renewcommand{\theequation}
{3.\arabic{equation}}

\setcounter{theorem}{0}
\renewcommand{\thetheorem}
{3.\arabic{theorem}}

\subsection{Basic constructions}

Let $m,d\geqslant 1$ be two positive integers.
Suppose $W_{1}\cong W_{2}\cong\cdots \cong W_{m}\cong W$ and
$W_{1}^{*}\cong W_{2}^{*}\cong\cdots \cong W_{d}^{*}\cong W^{*}$ are $n$-dimensional vector spaces over $\F$.

For any $1\leqslant j\leqslant m$ and any $1\leqslant k\leqslant d$, we choose a  basis $X_{j}:=\{x_{j1},x_{j2},\dots,x_{jn}\}$ for $W_{j}$ and choose a  basis $Y_{k}:=\{y_{k1},y_{k2},\dots,y_{kn}\}$ for $W_{k}^{*}$ such that
any $X_{j}$ and $Y_{k}$ are dual bases to one another.
Then we have
\begin{eqnarray*}
\F[\W]&=& \F[W_{1}\oplus\cdots\oplus W_{m}\oplus W_{1}^{*}\oplus\cdots\oplus W_{d}^{*}]\\
&\cong&\F[x_{j1},x_{j2},\dots,x_{jn},y_{k1},y_{k2},\dots,y_{kn}:1\leqslant j\leqslant m,1\leqslant k\leqslant d]
\end{eqnarray*}
 which is a polynomial algebra of Krull dimension $(m+d)n$.

For each pair $(j,k)$ with $1 \leq j \leq m$ and $1 \leq k \leq d$ there is an invariant
$\sum_{i=1}^n x_{ji} y_{ki}$ corresponding to $u_0 \in \F[\WW]^\GL$.  Similarly for each $(j,k)$ there are
invariants corresponding to $u_i$ for all $i$.  However, in the proof for the general cases we will work
exclusively with those invariants associated to a pair $(j,1)$ or $(1,k)$.  We introduce
notation for the invariants associated to these pairs as follows.

For $1\leqslant j\leqslant m$ and $i\in \N$, we define
\begin{eqnarray}
u_{ji}&:=&x_{j1}^{q^{i}}\cdot y_{11}+x_{j2}^{q^{i}}\cdot y_{12}+\cdots+x_{jn}^{q^{i}}\cdot y_{1n}. \\
u_{j,-i}&:=&x_{j1}\cdot y_{11}^{q^{i}}+x_{j2}\cdot y_{12}^{q^{i}}+\cdots+x_{jn}\cdot y_{1n}^{q^{i}}.
\end{eqnarray}
For $1\leqslant k\leqslant d$ and $i\in \N$, we define
\begin{eqnarray}
v_{ki}&:=&y_{k1}^{q^{i}}\cdot x_{11}+y_{k2}^{q^{i}}\cdot x_{12}+\cdots+y_{kn}^{q^{i}}\cdot x_{1n}. \\
v_{k,-i}&:=&y_{k1}\cdot x_{11}^{q^{i}}+y_{k2}\cdot x_{12}^{q^{i}}+\cdots+y_{kn}\cdot x_{1n}^{q^{i}}.
\end{eqnarray}
Note that $u_{10}=v_{10}$.

For any  $1\leqslant j\leqslant m$ and any $1\leqslant k\leqslant d$, Dickson's Theorem \cite{Dic1911}
yields
\begin{equation}
\label{ }
\F[W_{j}]^{\GL}=\F[x_{j1},x_{j2},\dots,x_{jn}]^{\GL}=\F[c_{j0},c_{j1},\dots,c_{j,n-1}]\ .
\end{equation}

The polynomial ring $\F[W_{j}\oplus W_{k}^{*}]$ can be endowed with an involutive algebra endomorphism
$$*_{jk}:\F[W_{j}\oplus W_{k}^{*}]\ra\F[W_{j}\oplus W_{k}^{*}], ~~f\mapsto f^{*}$$  determined  by
$x_{j1}\mapsto y_{kn}, x_{j2}\mapsto y_{k,n-1},\dots,x_{j,n-1}\mapsto y_{k2},x_{jn}\mapsto y_{k1}.$
As in the  previous section, we have
$\F[W_{k}^{*}]^{\GL}=\F[y_{k1},y_{k2},\dots,y_{kn}]^{\GL}=\F[c_{k0}^{*},c_{k1}^{*},\dots,c_{k,n-1}^{*}]$
where $c_{ki}^{*}=*_{jk}(c_{ji})$ for $0\leqslant i\leqslant n-1$.
 Moreover, Mui's Theorem \cite{Mui1975} yields
\begin{equation}
\label{ }
\F[W_{j}]^{\U}=\F[x_{j1},x_{j2},\dots,x_{jn}]^{\U}=\F[f_{j1},f_{j2},\dots,f_{jn}]
\end{equation}
where
\begin{equation}
\label{ }
f_{ji}=\prod_{v\in W_{j,i-1}^{*}} (x_{ji}+v)
\end{equation}
and $W_{j,i-1}^{*}$ denotes the subspace of $W_{j}^{*}$ with the basis $\{x_{j1},x_{j2},\dots,x_{j,i-1}\}$.
Similarly, we have
$$\F[W_{k}^{*}]^{\U}=\F[y_{k1},y_{k2},\dots,y_{kn}]^{\U}=\F[f_{k1}^{*},f_{k2}^{*},\dots,f_{kn}^{*}]$$
such that $f_{ki}^{*}=*_{jk}(f_{ji})$ for $1\leqslant i\leqslant n$.

\subsection{An application of Galois theory}

The following result directly generalizes Bonnaf\'e-Kemper \cite[Proposition 1.1]{BK2011}.

\begin{prop}\label{pre}
Let $G\leqslant\GL$ be any subgroup. Write $\F[W_{j}]^{G}=\F[g_{j1},\dots,g_{js}]$ and $\F[W_{k}^{*}]^{G}=\F[h_{k1},\dots,h_{kt}]$
for $1\leqslant j\leqslant m$ and $1\leqslant k\leqslant d$. Then
$\F(\W)^{G}$ is generated by $\{g_{j1},\dots,g_{js},u_{j0}:1\leqslant j\leqslant m\}\cup\{h_{k1},\dots,h_{kt}:1\leqslant k\leqslant d\}\cup\{v_{k0}:2\leqslant k\leqslant d\}$.
\end{prop}

\begin{proof}
The group $\overline{G} := \underbrace{G \times G \times \cdots \times G}_{m+d}$
acts on $\W$ in the obvious way.   Furthermore
$\F[\W]^{\overline{G}} = \F[g_{j1},\dots,g_{js},h_{k1},\dots,h_{kt}:1\leqslant j\leqslant m,1\leqslant k\leqslant d].$
For $\sigma \in G$ we write $\diag(\sigma)$ to denote the element
$\diag(\sigma) := (\sigma,\sigma,\dots,\sigma) \in G \subset \overline{G}$.

Let $L$ be the subfield of $\F(\W)$ generated by $$\{g_{j1},\dots,g_{js},u_{j0}:1\leqslant j\leqslant m\}\cup\{h_{k1},\dots,h_{kt}:1\leqslant k\leqslant d\}\cup\{v_{k0}:2\leqslant k\leqslant d\}$$ over $\F$.

Let $L_0$ denote the field
$$L_{0}:=\F(\W)^{\overline{G}}=\F(g_{j1},\dots,g_{js},h_{k1},\dots,h_{kt}:1\leqslant j\leqslant m,1\leqslant k\leqslant d).$$
By Artin's theorem \cite[page 264, Theorem 1.8]{Lang2002}, $\F(\W)$ is Galois over $L_{0}$ with group $\overline{G}$.
Thus $\F(\W)$ is also Galois over $L$ with the Galois group $G_{L}$, say. Now we have the following
situation:
$$
\xymatrix{
L_{0} ~\ar@{--}[d] \ar@{^{(}->}[r] &~L~ \ar@{--}[d] \ar@{^{(}->}[r]& ~\F(\W)^{G}~ \ar@{--}[d] \ar@{^{(}->}[r]&  ~\F(\W) \ar@{--}[d]\\
\overline{G}~& ~G_{L}~\ar@{_{(}->}[l]  &~G~ \ar@{_{(}->}[l] & ~1 \ar@{_{(}->}[l]
}$$
By Galois theory, to show $\F(\W)^{G}=L$, it suffices to show that $G_{L}=G$.
By definition $G\subseteq G_{L}$.

Conversely,
let $\alpha=(\sig_{1},\cdots,\sig_{m}, \tau_{1},\cdots,\tau_{d})$ be an arbitrary element of  $G_{L}\subseteq \overline{G}$.
   Fix $j$ with $1 \leq j \leq m$.
Since $u_{j0} \in L$ and $\diag(\sigma_j^{-1}) \in G \subset G_L$ and $\alpha \in G_L$, both $\alpha$ and $\diag(\sigma_j^{-1})$
 fix $u_{j0}$.

Thus we have
\begin{align*}
u_{j0}~=&~\alpha \cdot( \diag(\sigma_j^{-1}) \cdot u_{j0}) = (\alpha \,\diag(\sigma_j^{-1}))\cdot u_{j0}\\
          =&~(\sigma_1\sigma_j^{-1},\dots,\id,\dots,\sigma_m\sigma_j^{-1},\tau_1 \sigma_j^{-1}, \dots, \tau_d \sigma_j^{-1})\cdot u_{j0}\ .
\end{align*}
Since $\{x_{j1},x_{j2},\dots,x_{jn}\}$ is an algebraically independent set, this forces $(\tau_1 \sigma_j^{-1}) \cdot y_{1i} = y_{1i}$
for all $i=1,2,\dots,n$.  Therefore $\tau_1=\sigma_j$ and this holds for all $j=1,2,\dots,m$.

Similarly since $\alpha$ and $\diag(\tau_k^{-1})$ both fix $v_{k0}$ we have $\sigma_1 = \tau_k$ for all $k=1,2,\dots,d$.
Therefore, $\alpha = \diag(\sigma_1) \in G$ as required.
\end{proof}

Proposition~\ref{pre} immediately yields some large generating sets for $\F(\W)^G$ for $G=\GL$, $\SL$ or $\U$:

\begin{coro}\label{prec}
\begin{enumerate}
  \item The invariant field  $\F(\W)^{\GL}$ is generated by $$\{c_{j0},c_{j1},\dots,c_{j,n-1},u_{j0}:1\leqslant j\leqslant m\}\cup\{c_{k0}^{*},c_{k1}^{*},\dots,c_{k,n-1}^{*}:1\leqslant k\leqslant d\}\cup\{v_{k0}:2\leqslant k\leqslant d\}.$$
  \item  The invariant field $\F(\W)^{\SL}$ is generated by $$\{d_{jn},c_{j1},\dots,c_{j,n-1},u_{j0}:1\leqslant j\leqslant m\}\cup\{d_{kn}^{*},c_{k1}^{*},\dots,c_{k,n-1}^{*}:1\leqslant k\leqslant d\}\cup\{v_{k0}:2\leqslant k\leqslant d\}$$
  where $d_{jn}$ and $d_{kn}^{*}$ are defined by $c_{j0}=d_{jn}^{q-1}$ and $c_{k0}^{*}=(d_{kn}^{*})^{q-1}$ respectively.
  \item The invariant field $\F(\W)^{\U}$ is generated by $$\{f_{j1},\dots,f_{jn},u_{j0}:1\leqslant j\leqslant m\}\cup\{f_{k1}^{*},\dots,f_{kn}^{*}:1\leqslant k\leqslant d\}\cup\{v_{k0}:2\leqslant k\leqslant d\}.$$
\end{enumerate}
\end{coro}

\subsection{Generators for $\F(\W)^{\GL}$ and $\F(\W)^{\SL}$}

We define
\begin{eqnarray*}
A & := & \{c_{10},u_{1i}:1-n\leqslant i\leqslant n-1\} \\
B & := & \{u_{ji}:2\leqslant j\leqslant m,1-n\leqslant i\leqslant 0\}\\
C & := &  \{v_{ki}:2\leqslant k\leqslant d,0\leqslant i\leqslant n-1\}.
\end{eqnarray*}
Note that $|A|+|B|+|C|=2n+(m-1)n+(d-1)n=(m+d)n.$
The following theorem is the first main result.

\begin{thm} \label{GL}
The invariant field
$\F(\W)^{\GL}$ is generated by $A\cup B\cup C$. 
\end{thm}

\begin{proof}
Let $L:=\F(A\cup B\cup C)$ denote the subfield generated by $A\cup B\cup C$ over $\F$.
Since $A \cup B \cup C \subset \F[\W]^\GL$, it suffices to prove that
that $\F(\W)^{\GL}\subseteq L$.
By Corollary~\ref{prec} (1),
it is sufficient to show the following three statements:
\begin{align*}
\tag{$\dag_{1}$}\label{S1} c_{11},\dots, c_{1,n-1},c_{10}^{*},c_{11}^{*},\dots, c_{1,n-1}^{*} & \in  L \\
 \tag{$\dag_{2}$}\label{S2}   c_{j0},c_{j1},\dots, c_{j,n-1} & \in  L \qquad \textrm{ for }  2\leqslant j\leqslant m\\
  \tag{$\dag_{3}$}\label{S3}  c_{k0}^{*},c_{k1}^{*},\dots, c_{k,n-1}^{*} & \in  L \qquad \textrm{ for }  2\leqslant k\leqslant d.
\end{align*}

Firstly, consider the invariant field
$\F(W_{1}\oplus W_{1}^{\ast})^{\GL}=\F(x_{11},\dots,x_{1n},y_{11},\dots,y_{1n})^{\GL}$.
By Proposition~\ref{pGL},  $\F(W_{1}\oplus W_{1}^{\ast})^{\GL}=\F(c_{10},u_{1,1-n},\dots,u_{1,-1},u_{1,0},u_{11},\dots,u_{1,n-1})\subset L.$ Thus, $c_{11},\dots, c_{1,n-1},c_{10}^{*},c_{11}^{*},\dots, c_{1,n-1}^{*} \in  L$. This proves (\ref{S1}).

Secondly,
for $2\leqslant j\leqslant m$, we consider
$\F(W_{j}\oplus W_{1}^{\ast})^{\GL}=\F(x_{j1},\dots,x_{jn},y_{11},\dots,y_{1n})^{\GL}$.
It follows from Proposition~\ref{pGL}
that $$c_{j0},c_{j1},\dots, c_{j,n-1}\in \F(c_{j0},u_{j,1-n},\dots,u_{j,-1},u_{j0},u_{j1},\dots,u_{j,n-1})\subset L(c_{j0},u_{j1},\dots,u_{j,n-1}).$$
Thus to show (\ref{S2}) it suffices to show that
\begin{equation}
\label{ }
c_{j0},u_{j1},\dots,u_{j,n-1}\in L \quad\text{ for } 2 \leq j \leq m.
\end{equation}
Recall the relations (\ref{T1s}), (\ref{T2s}), and (\ref{Tn-1s}) in the invariant ring  $\F[W_{j}\oplus W_{1}^{\ast}]^{\GL}$ (see Chen-Wehlau \cite{CW2015}):
\begin{align*}
\tag{$T_{1}^{*}$}\label{T1s} c_{10}^{*}u_{j1}-c_{11}^{*}u_{j0}^{q}+c_{12}^{*}u_{j,-1}^{q}+\cdots+(-1)^{n}c_{1,n-1}^{*}u_{j,2-n}^{q} +(-1)^{n}u_{j,1-n}^{q} & =  0 \\
\tag{$T_{2}^{*}$}\label{T2s}     c_{10}^{*}u_{j2}-c_{11}^{*}u_{j1}^{q}+c_{12}^{*}u_{j0}^{q^{2}}+\cdots+(-1)^{n}c_{1,n-1}^{*}u_{j,3-n}^{q^{2}} +(-1)^{n}u_{j,2-n}^{q^{2}} & =  0   \\
\qquad\qquad\qquad\qquad\qquad\qquad\qquad\vdots\qquad\qquad\qquad\vdots\qquad\qquad\qquad\vdots\qquad\qquad\qquad\vdots\qquad\qquad&\quad\vdots\\
\tag{$T_{n-1}^{*}$}\label{Tn-1s}  c_{10}^{*}u_{j,n-1}-c_{11}^{*}u_{j,n-2}^{q}+c_{12}^{*}u_{j,n-3}^{q^{2}}-\cdots+(-1)^{n-1}c_{1,n-1}^{*}u_{j0}^{q^{n-1}}+(-1)^{n}u_{j,-1}^{q^{n-1}}&=0.
\end{align*}
Since $c_{10}^{*},c_{11}^{*},\dots, c_{1,n-1}^{*} \in  L$ by (\ref{S1}),
the relation (\ref{T1s}) implies that
$$u_{j1}\in L$$
which, together with the relation (\ref{T2s}), implies that
$$u_{j2}\in L.$$
Proceeding in this way and using the relations (\ref{T1s}), (\ref{T2s}), and (\ref{Tn-1s}) (in this order), we eventually obtain
\begin{equation}
\label{uj+}
u_{j1},u_{j2},\dots,u_{j,n-1}\in L.
\end{equation}
To show $c_{j0}\in L$, we recall that
$c_{j0}=d_{jn}^{q-1}$ and $c_{10}^{*}=(d_{1n}^{*})^{q-1}$.  Furthermore
\begin{eqnarray}
d_{jn}\cdot d_{1n}^{*}& = & \textrm{det}\begin{pmatrix}
     x_{j1} &x_{j2}&\cdots&x_{jn}    \\
     x_{j1}^{q} &x_{j2}^{q}&\cdots&x_{jn}^{q}   \\
     \vdots&\vdots&\vdots&\vdots\\
     x_{j1}^{q^{n-1}} &x_{j2}^{q^{n-1}}&\cdots&x_{jn}^{q^{n-1}}   \\
\end{pmatrix}\cdot\textrm{det}\begin{pmatrix}
     y_{11} &y_{11}^{q}&\cdots&y_{11}^{q^{n-1}}    \\
     y_{12} &y_{12}^{q}&\cdots&y_{12}^{q^{n-1}}    \\
     \vdots&\vdots&\vdots&\vdots\\
    y_{1n} &y_{1n}^{q}&\cdots&y_{1n}^{q^{n-1}}    \\
\end{pmatrix}\nonumber\\
&=&\textrm{det}\begin{pmatrix}
     u_{j0} &   u_{j,-1}&u_{j,-2}&\cdots&u_{j,1-n} \\
      u_{j1}&  u_{j0}^{q}&u_{j,-1}^{q}&\cdots& u_{j,2-n}^{q}\\
      u_{j2}&u_{j1}^{q}&\ddots&\ddots&\vdots\\
      \vdots&\ddots&\ddots&\ddots&u_{j,-1}^{q^{n-2}}\\
      u_{j,n-1}&u_{j,n-2}^{q}&\cdots&u_{j1}^{q^{n-2}}&u_{j0}^{q^{n-1}}
\end{pmatrix} \in L\ ,  \label{big equation}
\end{eqnarray}
 see Chen-Wehlau \cite{CW2015}.
 Hence $c_{j0}= (c_{10}^{*})^{-1}\cdot(c_{j0}\cdot c_{10}^{*}) = (c_{10}^{*})^{-1}\cdot(d_{jn}\cdot d_{1n}^{*})^{q-1}\in L$.
This completes the proof of (\ref{S2}).

Finally, we show (\ref{S3}). For any $2\leqslant k\leqslant d$, we consider
$$\F(W_{1}\oplus W_{k}^{\ast})^{\GL}=\F(x_{11},\dots,x_{1n},y_{k1},\dots,y_{kn})^{\GL}.$$
Therefore
$$c_{k0}^{*},c_{k1}^{*},\dots, c_{k,n-1}^{*}\in \F(c_{10},v_{k,1-n},\dots,v_{k,-1},v_{k0},v_{k1},\dots,v_{k,n-1})\subset L(v_{k,1-n},\dots,v_{k,-1}).$$ Thus it suffices to show that
\begin{equation}
\label{ }
v_{k,-1},v_{k,-2},\dots,v_{k,1-n}\in L.
\end{equation}
In the invariant  ring  $\F[W_{1}\oplus W_{k}^{\ast}]^{\GL}$, we have the following relations (\ref{T1}), (\ref{T2}), and (\ref{Tn-1})
described in \cite{CW2015}:
\begin{align*}
\tag{$T_{1}$}\label{T1} c_{10}v_{k,-1}-c_{11}v_{k0}^{q}+c_{12}v_{k1}^{q}+\cdots+(-1)^{n}c_{1,n-1}v_{k,n-2}^{q} +
(-1)^{n}v_{k,n-1}^{q} & =  0 \\
\tag{$T_{2}$}\label{T2}     c_{10}v_{k,-2}-c_{11}v_{k,-1}^{q}+c_{12}v_{k0}^{q^{2}}+\cdots+(-1)^{n}c_{1,n-1}v_{k,n-3}^{q^{2}} +(-1)^{n}v_{k,n-2}^{q^{2}} & =  0   \\
\cdots\quad\cdots\quad\cdots\quad\cdots\quad\cdots\quad\cdots\quad\cdots\quad\cdots\quad\cdots&\\
\tag{$T_{n-1}$}\label{Tn-1}  c_{10}v_{k,1-n}-c_{11}v_{k,2-n}^{q}+c_{12}v_{k,3-n}^{q^{2}}-\cdots+(-1)^{n-1}c_{1,n-1}v_{k0}^{q^{n-1}}+(-1)^{n}v_{k1}^{q^{n-1}}&=0.
\end{align*}
As before, these relations (\ref{T1}), (\ref{T2}), and (\ref{Tn-1}) (in this order), together with
(\ref{S1}), imply that $v_{k,-1},v_{k,-2},\dots,v_{k,1-n}\in L$.

This completes the proof.
\end{proof}

A  similar argument shows

\begin{thm} \label{SL}
The invariant field
$\F(\W)^{\SL}$ is  generated by $\{d_{1n},u_{1i}:1-n\leqslant i\leqslant n-1\}\cup\{u_{ji}:2\leqslant j\leqslant m,1-n\leqslant i\leqslant 0\}\cup \{v_{ki}:2\leqslant k\leqslant d,0\leqslant i\leqslant n-1\}$.
\end{thm}

\subsection{Generators for the invariant field $\F(\W)^{\U}$}

We define
\begin{eqnarray*}
D & := & \{f_{11},f_{12},f_{1s}^{*},u_{1i}:1\leqslant s\leqslant n-2, 2-n\leqslant i\leqslant 1\} \\
E &: = & \{f_{j1},u_{ji}:2\leqslant j\leqslant m,2-n\leqslant i\leqslant 0\}\\
F &:=& \{f_{k1}^{*},v_{ki}:2\leqslant k\leqslant d,0\leqslant i\leqslant n-2\}.
\end{eqnarray*}

\begin{thm}\label{UU}
\begin{enumerate}
  \item  For $n=1$, $\F(\W)^{\U}=\F(\W)$. 
  \item  For $n=2$, $\F(\W)^{\U}$ is generated by
  $$\{f_{11},f_{11}^{*},f_{12}^{*},u_{10}\}\cup\{f_{j1},u_{j0}:2\leqslant j\leqslant m\} \cup\{f_{k1}^{*},v_{k0}:2\leqslant k\leqslant d\}.$$
  \item For $n\geqslant 3$, $\F(\W)^{\U}$ is generated by $D\cup E\cup F$. 
\end{enumerate}
\end{thm}

\begin{proof}  The proof of (1) is immediate since $\U$ is the trivial group.

For (2) let $L$ denote the field generated by $\{f_{11},f_{11}^{*},f_{12}^{*},u_{10}\}\cup\{f_{j1},u_{j0}:2\leqslant j\leqslant m\}\cup\{f_{k1}^{*},v_{k0}:2\leqslant k\leqslant d\}$ over $\F$.
By Corollary~\ref{prec} (3),
$\F(\W)^{\U}$ is generated by $\{f_{j1},f_{j2},u_{j0}:1\leqslant j\leqslant m\}\cup\{f_{k1}^{*},f_{k2}^{*}:1\leqslant k\leqslant d\}\cup\{v_{k0}:2\leqslant k\leqslant d\}$. It suffices to show that
\begin{align*}
\tag{$\ddag_{1}$} \label{dd1} f_{j2}\in L    &\textrm{  for }1\leqslant j\leqslant m \\
 \tag{$\ddag_{2}$} \label{dd2}  f_{k2}^{*}\in L &  \textrm{  for }2\leqslant k\leqslant d.
\end{align*}
For any $1\leqslant j\leqslant m$, by the first equality in Proposition~\ref{pU} (2), we see that
$\F(W_{j}\oplus W_{1}^{*})^{\textrm{U}(2,\F)}=\F(x_{j1},x_{j2},y_{11},y_{12})^{\textrm{U}(2,\F)}=\F[f_{j1},f_{11}^{*},f_{12}^{*},u_{j0}]$.
Thus $f_{j2}\in \F(W_{j}\oplus W_{1}^{*})^{\textrm{U}(2,\F)}\subseteq L$, which proves (\ref{dd1}).
To show (\ref{dd2}), for any $2\leqslant k\leqslant d$, we consider $$\F(W_{1}\oplus W_{k}^{*})^{\textrm{U}(2,\F)}=\F(x_{11},x_{12},y_{k1},y_{k2})^{\textrm{U}(2,\F)}=\F[f_{11},f_{k1}^{*},f_{12},v_{k0}],$$ by the second equality in Proposition~\ref{pU} (2). This together with
 (\ref{dd1}) implies that $f_{k2}^{*}\in \F(W_{1}\oplus W_{k}^{*})^{\textrm{U}(2,\F)}\subseteq L$ for all $2 \leq k \leq d$.

 To prove (3) we first note that Corollary~\ref{prec} (3) implies $\F(\W)^{U}$ is generated by
 $\{f_{j1},\dots,f_{jn},u_{j0}:1\leqslant j\leqslant m\}\cup\{f_{k1}^{*},\dots,f_{kn}^{*}:1\leqslant k\leqslant d\}\cup\{v_{k0}:2\leqslant k\leqslant d\}$.
Let $L$ be the subfield generated by $D\cup E\cup F$  over $\F$. Note that $L\subseteq \F(\W)^{U}$ and $|D|+|E|+|F|=(m+d)n$.
Thus it is sufficient to show:
 \begin{align*}
 \tag{$\dag_{1}$} \label{d1} f_{13},\dots,f_{1n},f_{1,n-1}^{*},f_{1n}^{*} \in L,  &\\
\tag{$\dag_{2}$} \label{d2} f_{j2},f_{j3},\dots,f_{jn}  \in L    ,&\quad \textrm{  for }2\leqslant j\leqslant m, \\
 \tag{$\dag_{3}$} \label{d3}  f_{k2}^{*},f_{k3}^{*},\dots,f_{kn}^{*} \in L ,& \quad \textrm{  for }2\leqslant k\leqslant d.
\end{align*}
 By the first equality of Proposition~\ref{pU} (3), $ \F(W_{1}\oplus W_{1}^{*})^{\U}=\F(D).$
Thus $$f_{13},\dots,f_{1n},f_{1,n-1}^{*}, f_{1,n}^{*}   \in  \F(D)\subseteq L$$ which proves (\ref{d1}).

 To show (\ref{d2}), consider the invariant field $\F(W_{j}\oplus W_{1}^{*})^{\U}$ for any  $2\leqslant j \leqslant m$.  By the first equality of Proposition~\ref{pU} (3) again, we have seen that $\F(W_{j}\oplus W_{1}^{*})^{\U}$  is generated by $$\{f_{j1},f_{j2},f_{11}^{*},f_{12}^{*},\dots,f_{1,n-2}^{*},u_{j1},u_{j0},u_{j,-1},\dots,u_{j,2-n}\}.$$
 Thus $\{f_{j2},\dots,f_{jn}\}\subseteq\F(W_{j}\oplus W_{1}^{*})^{\U}\subseteq L(f_{j2},u_{j1})$.
 Recall the relations (\ref{Rj1+}) and (\ref{Rj2}) in $\F[W_{j}\oplus W_{1}^{*}]^{\U}$:
 \begin{align*}
\tag{$(R_1^{+})_{j1}$} \label{Rj1+}  f_{j1}^{q}\cdot f_{1n}^{*}&= u_{j,2-n}^{q}+\alp_{j,3-n}\cdot u_{j,3-n}^{q}+\dots+ \alp_{j0}\cdot u_{j0}^{q}+\alp_{j1}\cdot u_{j1}, \\
 & \quad\textrm{where all } \alp_{jk}\in\F[f_{11}^{*},\dots,f_{1,n-1}^{*}]\textrm{ are not zero}.\\
\tag{$(R_2)_{j1}$} \label{Rj2}  f_{j2}\cdot f_{1,n-1}^{*}&= u_{j,2-n}+\alp_{j,3-n}\cdot u_{j,3-n}+\dots+ \alp_{j0}\cdot u_{j0}\\
 &\quad +u_{j,3-n}^{q}+\bet_{j,4-n}\cdot u_{j,4-n}^{q}+\dots+ \bet_{j0}\cdot u_{j,0}^{q}+ \bet_{j1}\cdot u_{j1}, \\
 &\quad  \textrm{where all } \alp_{jk},\bet_{jk}\in\F[f_{j1},f_{j1}^{*},\dots,f_{j,n-2}^{*}] \textrm{ are not zero}.
  \end{align*}
 Since $\alp_{j1}\neq 0$ in (\ref{Rj1+}), we have $u_{j1}\in L$. This fact, together with
 (\ref{Rj2}), implies that $f_{j2}\in L$. Therefore
 $L(f_{j2},u_{j1})= L$ and $\{f_{j2},\dots,f_{jn}\}\subseteq L$.

To show (\ref{d3}),
we consider  $\F(W_{1}\oplus W_{k}^{*})^{\U}$ for any $2\leqslant k \leqslant d$.
 By the second equality of Proposition~\ref{pU} (3), it follows that $\F(W_{1}\oplus W_{k}^{*})^{\U}$ is generated by
$$\{f_{k1}^{*},f_{k2}^{*},f_{11},f_{12},\dots,f_{1,n-2},v_{k,-1},v_{k0},v_{k,1},\dots,v_{k,n-2}\}.$$
 Since $f_{k2}^{*},f_{k3}^{*},\dots,f_{kn}^{*}\in \F(W_{1}\oplus W_{k}^{*})^{\U}$, it suffices to show
 $\F(W_{1}\oplus W_{k}^{*})^{\U}\subseteq L$.
 It follows from (\ref{d1}) that $f_{13},\dots,f_{1,n-2}\in L$. Since $f_{k1}^{*},f_{11},f_{12},v_{k0},v_{k,1},\dots,v_{k,n-2}\in L$,
 we need only to show that
 $$f_{k2}^{*},v_{k,-1}\in L.$$
 As in the above proof of (\ref{d2}), we now use $*((R_1^+)_{k1})$ and $*((R_2)_{k1})$ to
  deduce that $(v_{k,-1})$ and $f_{k2}^{*}$ lie in $L$.
 This completes the proof.
\end{proof}

\section*{Acknowledgments}
We thank Gregor Kemper for helpful comments on an earlier draft of this article.
The first author thanks Queen's University at Kingston for providing a comfortable working environment during his visit in 2014--2016. The first author was supported by the Fundamental Research Funds for the Central Universities (2412017FZ001), NNSF of China (11401087) and CSC (201406625007).
Both authors were partially supported by NSERC.


\end{document}